\documentclass[a4paper,10pt]{amsart}
\usepackage{algorithmic}
\usepackage{algorithm}
\usepackage{mathtools}
\usepackage{pgfplots}

\newtheorem{theorem}{Theorem}
\newtheorem{definition}{Definition}
\newtheorem{lemma}{Lemma}
\newtheorem{remark}{Remark}
\newtheorem{proposition}[theorem]{Proposition}
\newtheorem{corollary}[theorem]{Corollary}

\newtheorem*{assumption*}{Assmuption}

\DeclareMathOperator{\ord}{ord}

\DeclareMathOperator{\Exp}{Exp}

\title{Isotopy equivalence of analytic branches in $(\mathbb{C}^n,0)$}
\author{P. Fortuny Ayuso}
\address{Dpto. Matemáticas, Universidad de Oviedo. Oviedo, Spain.}
\email{fortunypedro@uniovi.es}
\makeatletter
\@namedef{subjclassname@2020}{%
  \textup{2020} Mathematics Subject Classification}
\makeatother

\date{\today}

\usepackage{enumerate}

\definecolor{nuevo}{RGB}{0,0,120}

\begin{document}
\newcommand{\mex}[1]{\ensuremath{\left\lceil #1 \right\rceil}}
\begin{abstract}
  We prove that two analytic branches in $(\mathbb{C}^n,0)$ whose dual resolution graph is the same admit an ambient isotopy which is smooth outside the origin. A weaker version of the converse is also proved.
\end{abstract}
\maketitle
\section{Introduction}
The equivalence between equisingularity (in all its usual definitions) and topological equivalence of plane analytic branches over the complex plane is well-known since the works of Brauner \cite{Brauner}, Kh\"aler \cite{Kahler}, and Zariski \cite{Zariski-1965} (see, for example \cite{Wall} for a modern approach). For analytic curves in $(\mathbb{C}^n,0)$ there is no such result for any usual definition of equisingularity. As a matter of fact, the existence of many non-equivalent notions (examples due to Prof. Vicente Córdoba can be seen in \cite{Campillo}, and some more in \cite{Castellanos-2005}) seems to render this problem more complicated: what definition of equisingularity properly reflects the topology of the singularity?

In this brief note we give a partial answer to that question: two analytic singular curves in $(\mathbb{C}^n,0)$ whose resolution of singularities have the same dual graph are ambient isotopic, and the isotopy is smooth outside the singular point. The converse (the one we can prove) requires a technical condition which does not seem essential but is enough for our purposes: to obtain a combinatorial object which gives topological information on the singularity.

It is our conviction that the dual graph is a complete topological invariant, as in the planar case, but we do not master the required techniques to prove it.

\section{Setting and notation}
Let $\gamma=\gamma_0$ be a germ of analytic branch in $(\mathbb{C}^n,0)$, that is: a non-constant analytic map $\gamma:(\mathbb{C},0)\rightarrow (\mathbb{C}^n,0)$. The \emph{resolution of singularities of $\gamma$} is the (unique) non-empty finite sequence of point blowing-ups
\begin{equation}\label{eq:resolucion-of-gamma}
  \Pi\equiv \mathcal{X}_r \stackrel{\pi_r}{\longrightarrow}\mathcal{X}_{r-1}
  \stackrel{\pi_{r-1}}{\longrightarrow}\cdots\stackrel{\pi_2}{\longrightarrow}
  \mathcal{X}_1\stackrel{\pi_1}{\longrightarrow}(\mathbb{C}^n,0)=\mathcal{X}_{0}
\end{equation}
where $\pi_i$ is the blow up of $\mathcal{X}_{i-1}$ with center the center $P_{i-1}$ of the germ $\gamma_{i-1}$, and $\gamma_i=\pi^{-1}_i(\gamma_{i-1})$. For each $i$, $E_i=\pi_{i}^{-1}(P_{i-1})$ is the exceptional divisor of $\pi_i$. By definition, $r$ is the minimum integer such that $P_r$ is non-singular for $\gamma_r$ and transverse to the non-empty exceptional divisor $E=\Pi^{-1}(0)$, which is also non-singular at $P_r$.

Let $F$ be an irreducible component of the whole exceptional divisor $E$. We shall abuse notation and denote $F$ by $E_i$ in what follows if $\pi_{r}\circ\cdots\circ \pi_{i+1}(F)=E_i$ (i.e. $F$ ``appears'' when blowing-up $P_{i-1}$).

\begin{definition}\label{def:dual-graph}
  The \emph{dual graph} of $\gamma$ (or of $\Pi$) is the graph whose vertices $\mathcal{V}_{\gamma}$ are the irreducible components $E_i$ of the exceptional divisor $E=\Pi^{-1}(P_{0})$, and whose set of edges is:
  \begin{equation*}
    \mathcal{E}_{\gamma} = \left\{ (i,j) : E_i\cap E_j\neq \emptyset \right\}.
  \end{equation*}
\end{definition}

One may also turn to complexes and use trios $(i,j,k)$ when $E_i\cap E_j\cap E_k\neq \emptyset$ but in the case of branches and their resolution such an intersection is non-empty  if and only if $(i,j), (i,k)$ and $(j,k)$ belong to $\mathcal{E}_{\gamma}$, and thus we refrain from doing so.

From now on we work in $\mathbb{C}^{3}$ for simplicity, all the arguments carrying over to the general case without any difficulty.

Fix a set of coordinates $(x,y,z)$ in $(\mathbb{C}^3,0)$ and a possibly empty normal-crossings divisor $E$ whose equation is
\begin{equation*}
  E\equiv x^{\epsilon_1}y^{\epsilon_2}z^{\epsilon_3}=0
\end{equation*}
where $\epsilon_i\in \left\{ 0,1 \right\}$. Let $\gamma$ and $\eta$ be two non-singular curves with respective tangent vectors $(a_1,a_2,a_3)$, $(b_1,b_2,b_3)$ such that if $\epsilon_i=1$ then $a_i=0$ if and only if $b_i=0$ (that is, they have the same tangency relations with the irreducible components of $E$). The following lemma is the cornerstone of our results. 
\begin{lemma}\label{lem:motion-with-fixed-planes}
  With the above notations, there is a $\mathcal{C}^{\infty}$ vector field $X$ in $(\mathbb{C}^3,0)$ tangent to $E$, leaving $(0,0,0)$ fixed, which sends $\gamma$ to a curve tangent to $\eta$.
\end{lemma}
\begin{proof}
  Without loss of generality, given the hypotheses, we may assume that $a_1\neq 0$ so that $x$ is a parameter for the tangent lines $\dot{\gamma}=(x,\dot{a}_2(x),\dot{a}_3(x))$ and $\dot{\eta}=(x,\dot{b}_2(x),\dot{b}_3(x))$. Fix a determination of the logarithm. If $\ord_xa_i(x)=\ord_xb_i(x)=1$, then define $C_{i}(x)=\log(a_i(x)/b_i(x))$, and $c_i=C_i(0)$ otherwise either $\epsilon_i=1$ or, if $\epsilon_i=0$ then both $\ord_xa_i(x)$ and $\ord_xb_i(x)$ are at least $2$. In this latter case, set $c_{i}=0$. Let now:
  \begin{multline*}
    X = \bigg(
      0,
      (1-\epsilon_2)(b_2(x)-a_2(x)) + \epsilon_2 yc_2,
      (1-\epsilon_3)(b_3(x)-a_3(x)) + \epsilon_3 zc_3
    \bigg)
  \end{multline*}
  (where, by convention, $0\times K=0$ even if $K$ is not defined: this is to avoid useless repetitions). This is just a compact way of writing: \emph{the $i$-th coordinate of $X$ is equal to $b_i(x)-a_i(x)$ if $\epsilon_i=0$, and to $yc_i$ if $\epsilon_i=1$.}

  Let us show that $X$ satisfies the statement. For simplicity (the other cases follow exactly the same reasoning), we only consider the case $\epsilon_1=1$, $\epsilon_2=0$, $\epsilon_3=1$ (we do this case explicitly to convey the gist of the argument, as the two alternatives above are covered). Assume $\dot{a}_3(0)\neq 0$, so that $\dot{b}_3(0)\neq 0$ as well (if both are zero then the $z$-component of $X$ is $(z\times 0)=0$). The differential equation associated to $X$ is:
  \begin{equation}\label{eq:edo}
    \left\{
      \begin{array}{l}
        \dot{x} = 0\\[.5em]
        \dot{y} = b_2(x)-a_2(x)\\[.5em]
        \dot{z} = z \log(c_3)
      \end{array}
    \right.
  \end{equation}
  where $c_3=\lim_{x\rightarrow0}b_3(x)/a_3(x)$. The solutions of \eqref{eq:edo} for the initial condition $(x_{0},y_{0},z_{0})$ at time $1$ are:
  \begin{equation*}
    \Psi(x_0,y_0,z_0) = \left( x_{0},
      y_0+b_2(x_{0})-a_2(x_{0}),
      c_3z_{0}
    \right)
  \end{equation*}
  which sends the point $(x,a_2(x),a_3(x))$ to $(x,b_2(x),b_3(x))$. The fact that each irreducible component of $E$ is invariant for $X$ is obvious from the equations of $X$ and $E$, and also that $(0,0,0)$ is a fixed point of $X$.
\end{proof}

\begin{remark}
  Notice in the proof above that if $\ord_x(a_i(x)) = \ord_x(b_i(x)) = 1$, then we can always take the $i$-th component of $X$ to be $c_i(x)=\log(b_i(x)/a_i(x))$, and in this specific case, $X$ sends the $i$-th coordinate of $\gamma$ to the $i$-th coordinate of $\eta$ ``completely''.
\end{remark}

From this remark follows:
\begin{corollary}\label{cor:last-divisor}
  If $(P_i^1)=(P_i^2)$ for $i=0,\ldots, r$, then there is a vector field in a neighborhood $U$ of $P_r^1=P_r^2$ sending $\gamma_r$ to $\eta_r$ leaving $E_{r}\cap U$ invariant.
\end{corollary}
\begin{proof}
  Let $E_r\equiv x=0$ for simplicity. In the proof of Lemma \ref{lem:motion-with-fixed-planes}, we know that $a_j(x)$ and $b_j(x)$ are all parametrized by $x$, as both curves are non-singular and transverse to $E_r$. Let $X$ be
  \begin{equation*}
    X = (0, \log(b_2(x)/a_2 (x)), \log(b_3(x)/a_3(x)))
  \end{equation*}
  which is well defined because all the quotients are units, by transversality to $x=0$. This vector field sends $\gamma_r$ to $\eta_r$ in a neighborhood of $P_r$.
\end{proof}

\section{Same Dual Graph implies Isotopic Equivalence}
\begin{theorem}
  Two analytic branches at $(\mathbb{C}^n,0)$ which have the same dual resolution graph are ambient isotopic, and the isotopy can be taken to be smooth away from the origin.
\end{theorem}
\begin{proof}
  Let $\gamma^{1}$ and $\gamma^{2}$ be the branches in the statement. For a sequence of blow-ups like $\Pi$, $\gamma^1_i$ and $\gamma^2_i$ will represent their respective strict transforms at $\mathcal{X}_i$, and $P^1_i$, $P^2_i$ their intersection with the exceptional divisor (their respective infinitely near points).

  Both sequences $(P^{1}_i)$ and $(P_i^2)$ have the same length $r$ because the dual graphs are the same. Let $k$ be the first index such that $P^{1}_j=P^{2}_j$ for $j=0,\ldots, k-1$ and $P^1_k\neq P^2_k$. We reason by induction on $n=r-k$.

  Case $n=0$. Assume, for convenience, that $E_r=(x=0)$. As $\gamma^{1}_{r}$ and $\gamma^{2}_r$ are non-singular and transverse to $E_{r}$, Corollary \ref{cor:last-divisor} gives a vector field in an open neighborhood $U$ of $P_r$ sending $\gamma^1_r$ to $\gamma^2_r$. Let now $W$ be a closed ball $W\subset U$ and $Y$ be the null vector field in $V=\mathcal{X}\setminus \mathring{W}$. A partition of unity for the cover $\left\{ U,V \right\}$ gives rise to a vector field $Z$ which leaves the whole exceptional divisor $E$ invariant and sends $\gamma^1_{r}$ to $\gamma^2_{r}$. This proves the basis step, as $Z$ can be pulled-forward to $(\mathbb{C}^3,0)\setminus \left\{ 0 \right\}$. Its extension to the origin by the null vector is trivially continuous and we get the desired isotopy which is $\mathcal{C}^{\infty}$ outside the origin.

  Assume the result is true for $n-1=r-(k+1)$, and consider the case $n=r-k$. At $P_k$, the curves $\gamma^1_k$ and $\gamma^2_k$ have the same tangency relations with the exceptional divisor $E_k$ (otherwise they would not give rise to the same dual graph), but by hypothesis, their tangent lines are different. By Lemma \ref{lem:motion-with-fixed-planes} we can define a $\mathcal{C}^{\infty}$ vector field on a neighborhood of $P_k$ sending $\gamma^1_k$ to a curve $\overline{\gamma}^1_{k}$ tangent to $\gamma^2_k$. This vector field, by the same argument as before, can be extended to the whole $\mathcal{X}_k$. The curves $\gamma^1$ and $\overline{\gamma}^1:=\pi_k(\overline{\gamma}^1_k)$ are isotopic by the argument above, and the curves $\overline{\gamma}^{1}$ and $\gamma^2$ share the same infinitely near points up to $k+1$, so that they are also isotopic by induction. This completes the proof.
\end{proof}

The converse we can prove is weaker but possibly informative. The proof is essentially contained in the statement, and is done by induction on the number of shared infinitely near points.
\begin{proposition}
  Assume $\gamma$ and $\eta$ are two analytic branches in $(\mathbb{C}^n,0)$. If there is a sequence $(X_i)_{i=0}^{r}$ of $\mathcal{C}^{\infty}$ vector fields $X_i$ defined in $(\mathbb{C}^n,0)\setminus \left\{ 0 \right\}$ and branches $\tilde{\gamma}_{i}$ such that:
  \begin{enumerate}
  \item The branch $\tilde{\gamma}_i$ and $\gamma$ share the first $i$ infinitely near points $(P_j)_{j=0}^i$,
  \item Each pull back $\pi^{-1}(X_{i})$ can be extended in a neighborhood $U_{i}$ of $P_{i}\in E_i$ to a vector field $\tilde{X}_i$ leaving $E_i\cap U_i$ invariant,
  \item The flow $\Exp(\tilde{X}_i)$ in $U_i$ sends $\pi^{-1}_{i}(\tilde{\gamma}_i)$ to $\pi^{-1}_{i}(\tilde{\gamma}_{i+1})$,
  \item Finally, $\tilde{\gamma}_r=\eta$
  \end{enumerate}
  then $\gamma$ and $\eta$ have the same dual graph.
\end{proposition}
\begin{proof}
  We reason by induction on the length $r$ of the shorter resolution of singularities of $\eta$ or $\gamma$. If $r=0$ then the result is obvious because the statement means that there is a flow sending $\gamma$ to $\eta$ and that one of them is non-singular. Hence, the other must also be non-singular.

  Assuming the case $n\geq 0$ true, consider the case $n+1\geq 1$. This implies that $\gamma$ and $\eta$ are tangent at $0$. let $\gamma_1$ and $\eta_1$ be their strict transforms by $\pi_1$, which meet at $P_1$ and share $n$ infinitely near points. All the four conditions are met by $\gamma_1$ and $\eta_1$, so that they have the same dual graph. Moreover, if one is tangent to $E_1$, then so is the other, as flows respect tangency relations. This implies that $\gamma$ and $\eta$ have the same dual graph and the proof is concluded.
\end{proof}

\end{document}